\documentclass[letterpaper,11pt,reqno]{amsart}

\usepackage[utf8]{inputenc}
\usepackage{mathcommand}
\usepackage{xpatch}
\usepackage{xcolor}
\usepackage{xstring}
\usepackage[pagebackref=true]{hyperref}
\usepackage[T1]{fontenc}
\usepackage{fancyhdr}
\usepackage{geometry}
\usepackage{calligra}
\usepackage{amsmath}
\usepackage{amsthm}
\usepackage{thmtools}
\usepackage{amssymb}
\usepackage[
    cal=cm,
    scr=boondox,
    frak=boondox,
    bb=ams,
]{mathalpha}
\usepackage{mathtools}
\mathtoolsset{centercolon}
\usepackage{stmaryrd}
\SetSymbolFont{stmry}{bold}{U}{stmry}{m}{n}
\usepackage{mathdots}
\usepackage{mathabx}
\usepackage{mleftright}
\mleftright
\usepackage{physics}
\usepackage{embrac}
\usepackage{graphicx}
\usepackage{framed}
\usepackage[capitalize]{cleveref}
\usepackage{multicol}
\usepackage{paracol}
\usepackage{array}
\usepackage{tabularray}
\UseTblrLibrary{amsmath,booktabs,diagbox}
\usepackage{enumitem}
\usepackage{tikz}
\usetikzlibrary{cd,arrows.meta, positioning}
\usepackage{quiver}
\usepackage{microtype}

\definecolor{verydarkblue}{rgb}{0,0,0.5}

\hypersetup{
    breaklinks,
    colorlinks,
    citecolor=verydarkblue,
    linkcolor=verydarkblue,
    urlcolor=verydarkblue,
}

\theoremstyle{plain}
\newtheorem{theorem}{Theorem}[section]
\newtheorem{lemma}[theorem]{Lemma}
\newtheorem{proposition}[theorem]{Proposition}
\newtheorem{corollary}[theorem]{Corollary}
\newtheorem{conjecture}[theorem]{Conjecture}

\theoremstyle{definition}
\newtheorem{definition}[theorem]{Definition}

\newtheorem{example}[theorem]{Example}
\newtheorem{remark}[theorem]{Remark}

\theoremstyle{remark}
\numberwithin{equation}{section}

\makeatletter

\patchcmd{\section}{\normalfont\scshape}{\large\bfseries}{}{}
\patchcmd{\subsection}{\normalfont}{\bfseries}{}{}
\patchcmd{\subsubsection}{\normalfont}{\bfseries}{}{}
\def\@secnumfont{\@empty}

\patchcmd{\@settitle}{\uppercasenonmath\@title}{\LARGE}{}{}
\patchcmd{\@setauthors}{\MakeUppercase}{\large}{}{}
\patchcmd{\HyOrg@maketitle}{\uppercasenonmath\shorttitle}{}{}{}
\patchcmd{\HyOrg@maketitle}{\thispagestyle{firstpage}}{\thispagestyle{empty}}{}{}
\patchcmd{\@maketitle}{\@adminfootnotes}{{\def\@makefntext{\noindent\@makefnmark}\@adminfootnotes}}{}{}
\patchcmd{\@adminfootnotes}{\@setthanks}{\setlength\parindent{0pt}\@setthanks}{}{}

\patchcmd{\@setaddresses}{\def\email}{\def\email##1##2{\href{mailto:##2}{\texttt{##2}}}\def\pepe}{}{}
\patchcmd{\@setaddresses}{\\}{\pepe}{}{}
\patchcmd{\@setaddresses}{\par\addvspace\bigskipamount\indent}{}{}{}
\patchcmd{\@setaddresses}{(\ignorespaces}{\switchcolumn\vspace{1em}\begingroup\itshape}{}{}
\patchcmd{\@setaddresses}{\unskip) }{\unskip\endgroup\\}{}{}
\patchcmd{\@setaddresses}{\scshape}{}{}{}
\patchcmd{\@setaddresses}{\addresses}{{\vspace{2em}\begin{paracol}{2}\setlength{\parindent}{0pt}\switchcolumn\addresses\end{paracol}}}{}{}

\makeatother

\begin{document}

\title{On the Nash Problem over 3-Fold Terminal Singularities of Type cAx/2}

\author{Keng-Hung Steven Lin}

\address[Former address at OU]{%
    Department of Mathematics\\
    University of Oklahoma\\
    601 Elm Avenue, Room 423\\
    Norman, OK 73019 (USA)%
}

\email{mitmosfet@ou.edu}

\address[Current address at NTU-GICE]{%
    Graduate Institute of Communication Engineering\\
    National Taiwan University\\
    BL-501, No. 1, Sec. 4, Roosevelt Road, Taipei 10617 (Taiwan)%
}

\email{khslin@ntu.edu.tw}


\begin{abstract}
We study Nash valuations on 3-fold terminal singularities, especially in type cAx/2. We find that, in type cAx/2, exceptional prime divisors computing the minimal discrepancy (which is 1/2 in this case) induce Nash valuations. We conjecture this in general for all 3-fold terminal singularities, and provide some evidence in the Gorenstein case.  
\end{abstract}

\maketitle


\section{Introduction}

The arc space $X_{\infty}$ of a variety $X$ plays an important role in algebraic geometry \cite{EM09,CAS18}, especially in singularity theory. In \cite{Nash95}, J.\ F.\ Nash proposed using irreducible families of arcs based on the singularities $X_{\mathrm{sing}}$ in $X$, the so-called the \emph{Nash components} of $X$, to characterize essential valuations, which are the valuations appearing in every resolution of singularities. In other words, he shows the following map
\[
\mathcal{N}_{X}:\{\text{irreducible components of }\pi_{X}^{-1}(X_{\mathrm{sing}})\}\to\{\text{essential valuations of }X\}
\]is injective in \cite{Nash95}, where $\pi_{X}:X_{\infty}\to X$ denotes the canonical map from $X_{\infty}$ to $X$. The map $\mathcal{N}_{X}$ is called the \emph{Nash map} of $X$, and the valuation induced by an irreducible component of $\pi_{X}^{-1}(X_{\mathrm{sing}})$ is called a \emph{Nash valuation} of $X$. 
In addition, the injectivity of $\mathcal{N}_{X}$ can be reformulated as follows.
\[
\{\text{Nash valuations of }X\}\subseteq\{\text{essential valuations of }X\}.
\]

It is natural to ask if the Nash map is surjective, i.e., is any essential valuation of $X$ obtainable from a Nash valuation of $X$? It is known that Nash maps are always surjective on all curve or surface singularities \cite{FdBPP12}, and counterexamples appear when the dimension of $X$ is at least $3$ \cite{dF13,IK03}. Moreover, Nash maps are surjective for singularities coming from some special spaces \cite{IK03, GPPD07, PPP08, BLM25}.  

Rather than surjectivity, characterizing the image of the Nash map becomes a central question \cite{AJR06,LJR12,dFD16}, which is known as the \emph{Nash problem}. In \cite{dFD16}, T.~de~Fernex and R.~Docampo use the relatively minimal model program (rMMP) over $X$ to provide part of the image of the Nash map. More precisely, they introduce the \emph{terminal valuations} of $X$ and prove that they are Nash valuations of $X$.
\[
\{\text{terminal valuations of }X\}\subseteq\{\text{Nash valuations of }X\}
\]However, this result becomes useless when $X$ has only terminal singularities. In particular, there are no terminal valuations in this case, so we need other characterizations for their Nash valuations.

We aim to use discrepancies to characterize Nash valuations when $X$ has only terminal singularities. In this paper, we focus on the $3$-fold terminal singularities of type cAx/2, and obtain the following main theorem.

\begin{theorem}\label{mainthmcax2}
Let $(X,p)$ be a germ of $3$-fold terminal singularity of type cAx/2. Then any exceptional prime divisor $E$ over $X$ with minimal discrepancy (i.e., $a(E,X)=\tfrac{1}{2}$) induces a Nash valuation of $X$.
\end{theorem}
This theorem is motivated by results for toric terminal singularities. In this paper, we use discrepancies to reformulate the result of \cite{IK03} and obtain the following statement.
\begin{theorem}[\cite{IK03}]\label{mainthm1}
    Let $X(\Delta)$ be a $\mathbb{Q}$-Gorenstein toric variety associated to a fan $\Delta\subseteq N_{\mathbb{R}}$ of dimension a least $3$ having only terminal singularities. Then any exceptional prime divisor $E$ over $X(\Delta)$ with discrepancy $a(E,X(\Delta))\leq 1$ induces a Nash valuation $\mathrm{val}_{E}$ of $X(\Delta)$. 
\end{theorem}

Based on these and the results of \cite{IK03,JK13,Chen23}, we propose the following conjectures.
\begin{conjecture}[Conjecture A]\label{conjA}
Let $(X,p)$ be a 3-fold terminal singularity with Gorenstein index $m\geq 1$. Then every exceptional prime divisor $E$ over $X$ with $a(E,X)\leq 1$ induces a Nash valuation $\mathrm{val}_{E}$ of $X$.    
\end{conjecture}

\begin{conjecture}[Conjecture B]\label{conjB}
Let $(X,p)$ be a 3-fold terminal singularity with Gorenstein index $m\geq 1$. Then every exceptional prime divisor $E$ over $X$ with minimal discrepancy (i.e.,  $a(E,X)=\frac{1}{m}$) induces a Nash valuation $\mathrm{val}_{E}$ of $X$.    
\end{conjecture}
In other words, we show that Conjecture A holds for terminal toric singularities in any dimension, and prove that Conjecture B holds for 3-fold terminal singularities of type cAx/2. 

For the case of Gorenstein terminal singularities, we apply a similar argument to Theorem~\ref{mainthmcax2} and give the following partial result, which provides evidence for our conjectures.

\begin{theorem}
Let $(X,p)$ be a germ of $3$-fold Gorenstein terminal singularity. Let $\pi:\widebar{X}\to X$ be a divisorial contraction with minimal discrepancy (it is $1$ in this case), and let $E=\pi^{-1}(p)$ be its exceptional prime divisor. Assume that the singularities of $\widebar{X}$ are terminal cyclic quotient or Gorenstein terminal. Then there is an exceptional prime divisor over $X$ with discrepancy $1$ that induces a Nash valuation. In other words, 
\[
\{F \text{ exceptional over }X\ |\ a(F,X)=1 \text{ and } \mathrm{val}_{F} \text{ is a Nash valuation of } X\}\neq\emptyset.
\]Moreover, if every exceptional prime divisor over $X$ computing the minimal discrepancy can be obtained by a divisorial contraction with minimal discrepancy to $X$, then $\mathrm{val}_{E}$ is a Nash valuation of $X$.    
\end{theorem}


\section{Preliminaries}

\subsection{3-Fold Terminal Singularities}

We discuss our settings and introduce the notation for $3$-fold terminal singularities. We fix the base field to be the complex numbers $\mathbb{C}$. A variety $X$ is an integral separated scheme over $\mathbb{C}$. All of our varieties will be assumed to be normal.
 
Let $X$ be $\mathbb{Q}$-Gorenstein variety with singular locus $X_{\mathrm{sing}}$, i.e., there is an $m\in\mathbb{Z}_{>0}$ such that $mK_{X}$ is Cartier, where $K_{X}$ is a canonical divisor for $X$. The smallest such $m$ is called the \emph{Gorenstein index} of $X$. Let $f:Y\to X$ be a resolution of singularities such that the exceptional locus $\mathrm{exc}(f)$ is a divisor (for example, a log resolution). A divisor $E$ is \emph{over} $X$ if there is a birational morphism $f:Y\to X$ such that $E\subseteq Y$. A divisor $E$ over $X$ is \emph{exceptional} if,  additionally, $E\subseteq\mathrm{exc}(f)$.

Consider the \emph{relatively canonical divisor} $K_{Y/X}:=K_{Y}-f^{*}K_{X}$, which is chosen to be supported on $\mathrm{exc}(f)$. We have 
\[
K_{Y/X}=\sum_{i\in I}a(E_{i},X)E_{i}
\]where $\{E_{i}\}_{i\in I}$ is the set of exceptional prime divisors of $f$. The number $a(E,X)\in\mathbb{Q}$ is called the \emph{discrepancy} of $X$ with respect to the divisor $E$ over $X$. The \emph{minimal discrepancy} of $X$ is defined as the infimum of the discrepancies $a(E,X)$, where $E$ runs through all exceptional prime divisors over $X$. Notice that the number $a(E,X)$ is independent of the choice of resolution. We say $X$ has only \emph{terminal singularities} (or $X$ is \emph{terminal}, briefly) if $a(E,X)>0$ for every exceptional prime divisor $E$ over $X$. 

Notice that $\mathrm{codim}(X_{\mathrm{sing}},X)\geq 3$ when $X$ is terminal \cite{KM98}, so $X_{\mathrm{sing}}$ is the set of isolated singular points when $X$ is a terminal $3$-fold. For a germ of $3$-fold terminal singularity $(X,p)$, the cases for which the Gorenstein index is equal to $1$ (or \emph{Gorenstein terminal}, shortly) are characterized by M.\ Reid in \cite{Reid83}.
\begin{theorem}[\cite{Reid83}]
A germ of 3-fold Gorenstein singularity is terminal if and only if it is an isolated $cDV$ point.  
\end{theorem}

For the cases for which the Gorenstein index greater than $1$, the complete classification is treated in \cite{Dan83, MS84, Mori85}. 

Let us introduce the classic notation to describe the classification. Denote the coordinates of the complex space $\mathbb{C}^{4}$ by $(x,y,z,u)$. Let $\mathbb{Z}_{m}$ be a cyclic group of order $m$, and consider the action $\tau:\mathbb{Z}_{m}\times\mathbb{C}^{4}\to\mathbb{C}^{4}$ by $\tau(x)=\zeta^{\alpha_{1}}x,\tau(y)=\zeta^{\alpha_{2}}y,\tau(z)=\zeta^{\alpha_{3}}z,\tau(u)=\zeta^{\alpha_{4}}u$ where $\zeta$ is a primitive $m$-root of unity and $\alpha_{i}\in\mathbb{Z}$ for $i=1,2,3,4$. The quotient variety of $\mathbb{C}^{4}$ with respect to this action is denoted by 
\[
\mathbb{C}^{4}/\tfrac{1}{m}(\alpha_{1},\alpha_{2},\alpha_{3},\alpha_{4}).
\]Its ring of regular functions agrees with the ring of invariant regular functions with respect to $\frac{1}{m}(\alpha_{1},\alpha_{2},\alpha_{3},\alpha_{4})$ (see \cite[Chapter 5]{CLS11}).

Denote the ring of complex convergent power series in the variables $x,y,z,u$ by $\mathbb{C}\{x,y,z,u\}$. Let $\phi\in\mathbb{C}\{x,y,z,u\}$ be a $\mathbb{Z}_{m}$-semi-invariant with respect to the previous action determined by the vector $\frac{1}{m}(\alpha_{1},\alpha_{2},\alpha_{3},\alpha_{4})$. Then we obtain an induced action on the germ of hypersurface $(\phi=0)\subseteq\mathbb{C}^{4}$. It induces a germ of subvariety of $\mathbb{C}^{4}/\frac{1}{m}(\alpha_{1},\alpha_{2},\alpha_{3},\alpha_{4})$, and it is denoted by $(\phi=0)\subseteq\mathbb{C}^{4}/\frac{1}{m}(\alpha_{1},\alpha_{2},\alpha_{3},\alpha_{4})$. It is called a hyperquotient singularity.

With this notation, we have the following list for the classification:

\begin{theorem}[\cite{Dan83,MS84,Mori85,Hay99,Hay00,Chen16}]\label{morilist}
Let $X$ be a germ of 3-fold terminal singularity with Gorenstein index greater than $1$. Then it is analytically isomorphic to one of the following forms: 
\begin{itemize}
\item (type cA/m) 
\[
(xy+f(z,u)=0)\subseteq\mathbb{C}^{4}/\tfrac{1}{m}(a,-a,1,0)
\]where $a,m\in\mathbb{Z}_{>0}$ with $\gcd(a,m)=1$ and $f(z,u)\in\mathbb{C}\{z,u\}$ is a $\mathbb{Z}_{m}$-invariant.
\item (type cAx/2) \[
(x^{2}+y^{2}+f(z,u)=0)\subseteq\mathbb{C}^{4}/\tfrac{1}{2}(0,1,1,1)
\]where $f(z,u)\in(z,u)^{4}\subseteq\mathbb{C}\{z,u\}$ is a $\mathbb{Z}_{2}$-invariant.
\item (type cAx/4) \[
(x^{2}+y^{2}+f(z,u)=0)\subseteq\mathbb{C}^{4}/\tfrac{1}{4}(1,3,1,2)
\]where $f(z,u)\in\mathbb{C}\{z,u\}$ is a $\mathbb{Z}_{4}$-semi-invariant such that $f(z,u)$ does not contain $u$ as a monomial.
\item (type cD/2) 
\[
(\phi=0)\subseteq\mathbb{C}^{4}/\tfrac{1}{2}(1,1,0,1)
\]where $\phi$ has one of the following forms:
\begin{itemize}
    \item (cD/2-1) $\phi=u^{2}+xyz+x^{2a}+y^{2b}+z^{c}$ where $a,b\geq 2$, $c\geq 3$.
    \item (cD/2-2) $\phi=u^{2}+y^{2}z+\lambda yx^{2a+1}+g(x,z)$ where $\lambda\in\mathbb{C}$, $a\geq 1$, \\ $g(x,z)\in(x^{4},x^{2}z^{2},z^{3})\subseteq\mathbb{C}\{x,z\}$.
\end{itemize}
\item (type cD/3) 
\[
(\phi=0)\subseteq\mathbb{C}^{4}/\tfrac{1}{3}(1,2,2,0)
\]where $\phi$ has one of the following forms:
\begin{itemize}
    \item (cD/3-1) $\phi=x^{3}+u^{2}+yz(y+z)$.
    \item (cD/3-2) $\phi=x^{3}+u^{2}+yz^{2}+xy^{4}\alpha(y^{2})+y^{6}\beta(y^{3})$\\  where $\alpha(y^{3}),\beta(y^{3})\in\mathbb{C}\{y^{3}\}$ with $4\alpha^{3}+27\beta^{2}\neq 0$.
    \item (cD/3-3) $\phi=x^{3}+u^{2}+y^{3}+xyz^{3}\alpha(z^{3})+xz^{4}\beta(z^{3})+yz^{5}\gamma(z^{3})+z^{6}\delta(z^{3})$\\ where $\alpha(z^{3}),\beta(z^{3}),\gamma(z^{3}),\delta(z^{3})\in\mathbb{C}\{z^{3}\}$.
\end{itemize}
\item (type cE/2) 
\[(x^{3}+u^{2}+g(y,z)+h(y,z)=0)\subseteq\mathbb{C}^{4}/\tfrac{1}{2}(0,1,1,1)
\]where $g(y,z)\in(y,z)^{4}\subseteq\mathbb{C}\{y,z\}$, $h(y,z)\in(y,z)^{4}\setminus(y,z)^{5}\subseteq\mathbb{C}\{y,z\}$.
\end{itemize}
\end{theorem}
Three-fold toric terminal singularities are studied in \cite{Dan83} and \cite{MS84}. The classification is completed by the following result in \cite{KSB88}. 
\begin{theorem}[{\cite[Theorem 6.5]{KSB88}}]
Let
\[
X=(\phi(x,y,z,u)=0)\subseteq\mathbb{C}^{4}/\tfrac{1}{m}(\alpha_{1},\alpha_{2},\alpha_{3},\alpha_{4})
\]
be one of the hyperquotient singularities listed in Theorem \ref{morilist}. Assume:
\begin{itemize}
    \item $\phi(x,y,z,u)\in\mathbb{C}[x,y,z,u]$ defines an isolated singularity at $(0,0,0,0)\in\mathbb{C}^{4}$.
    \item The action $\tau:\mathbb{Z}_{m}\times\mathbb{C}^{4}\to\mathbb{C}^{4}$ defined previously is free outside $(0,0,0,0)\in\mathbb{C}^{4}$.
\end{itemize}
Then $(X,(0,0,0,0))$ defines a 3-fold terminal singularity.
\end{theorem}
Three-fold terminal singularities have been studied explicitly in many articles \cite{Hay99,Hay00,HayG05,Chen16}. One of the important tools for explicitly studying them is the \emph{divisorial contraction with minimal discrepancy} (or \emph{divisorial blowing-ups with minimal discrepancy}).

\begin{theorem}[\cite{Kaw93, Mar96}]
For a 3-fold terminal singularity with Gorenstein index $m\geq 1$, its minimal discrepancy is exactly $\frac{1}{m}$.
\end{theorem}

\begin{definition}[\cite{Hay99} and \cite{Hay00}]\label{def:HayDivCon}
Let $(X,p)$ be a germ of 3-fold terminal singularity with Gorenstein index $m\geq 2$. A \emph{divisorial contraction to $X$ with minimal discrepancy} is a projective birational morphism $\pi:\widebar{X}\to X$ such that
\begin{itemize}
    \item $\widebar{X}$ has only terminal singularities,
    \item the exceptional locus of $\pi$ is a prime divisor $E$ over $X$,
    \item and $K_{\widebar{X}/X}=\frac{1}{m}E$, i.e.,\ $E$ computes the minimal discrepancy $a(E,X)=\tfrac{1}{m}$ of $X$.
\end{itemize}
\end{definition}

\begin{remark}
In Definition \ref{def:HayDivCon}, we see that $-K_{\overline{X}}$ is $\pi$-ample.     
\end{remark}

The existence of such morphisms is proved by T.\ Hayakawa. 

\begin{theorem}[\cite{Hay99,Hay00}]
Let $(X,p)$ be a germ of 3-fold terminal singularity such that its Gorenstein index is greater than $1$.
Then there is a divisorial contraction to $X$ with minimal discrepancy.    
\end{theorem}

\subsection{The Arc Space of a Variety}
We briefly recall some important concepts and notation about the arc space $X_{\infty}$ of a variety $X$. For more details, see \cite{IK03, CAS18, dFD16,dFarcsp18}.  

J.F.\ Nash first introduces the concept of arc spaces to study singularities of a variety $X$ in \cite{Nash95}. Generally speaking, an arc on a variety $X$ is a formal analytic curve in $X$, and the arc space $X_{\infty}$ is the space which parametrizes all formal analytic curves in $X$. 

Let $X$ be a scheme of finite-type over $\mathbb{C}$, and let $\mathbb{K}\supseteq\mathbb{C}$ be a field extension of $\mathbb{C}$. A $\mathbb{K}$-\emph{arc} in $X$ is a morphism $\alpha:\mathrm{Spec}(\mathbb{K}[\![t]\!])\to X$. 
In many scenarios, we briefly say an \emph{arc} of $X$ rather than a $\mathbb{K}$-arc if it is unnecessary to specify the base $\mathbb{K}$. We denote the image of the closed point of $\mathrm{Spec}(\mathbb{K}[\![t]\!])$ by $\alpha(0)$, and the image of the generic point of $\mathrm{Spec}(\mathbb{K}[\![t]\!])$ by $\alpha(\eta)$. 

For $m\in\mathbb{Z}_{\geq0}$, the functor
\[
X_{m}(-):=\mathrm{Hom}_{\mathbb{C}\text{-}\mathrm{Sch}}(\mathrm{Spec}(-[t]/(t)^{m+1}),X):(\mathbb{C}\text{-}\mathrm{Algebras })\to(\mathrm{Sets})
\]is representable by a scheme of finite-type over $\mathbb{C}$. The representative is denoted by $X_{m}$, and it is called the \emph{m-th jet scheme} of $X$. For a $\mathbb{C}$-algebra $A$ and $m,n\in\mathbb{Z}_{\geq 0}$ with $m>n$, the natural quotient maps of algebras $A[t]/(t^{m+1})\to A[t]/(t^{n+1})$ induce the affine morphisms $\pi_{m,n}:X_{m}\to X_{n}$, called the \emph{truncation morphisms} of jet schemes. $\{(X_{m},\pi_{m,n})\}$ forms an inverse system of schemes over $\mathbb{C}$.    

The \emph{arc space} $X_{\infty}$ of $X$ is defined as the inverse limit of the jet schemes of $X$.
\[
X_{\infty}:=\varprojlim X_{m}.
\]For a field extension $\mathbb{K}\supseteq\mathbb{C}$, the set of $\mathbb{K}$-valued points of $X_{\infty}$
\[
X_{\infty}(\mathbb{K}):=\mathrm{Hom}_{\mathbb{C}\text{-}\mathrm{Sch}}(\mathrm{Spec}(\mathbb{K}[\![t]\!]),X)
\]
is the set of $\mathbb{K}$-arcs of $X$.

The $X_{\infty}$ is equipped with a canonical projection $\pi_{\infty,m}:X_{\infty}\to X_{m}$. For $m=0$, we obtain a morphism $\pi_{X}:=\pi_{\infty,0}$.
\begin{equation*}
\begin{split}
\pi_{X}:X_{\infty}&\longrightarrow X \\
\alpha&\longmapsto \alpha(0).
\end{split}
\end{equation*}

\subsection{The Nash Problem}
We will formulate the \emph{Nash problem}, which is an interesting observation proposed by J.\ F.\ Nash in \cite{Nash95}. From his point of view, the set of all arcs passing through the singular locus $X_{\mathrm{sing}}$ of $X$ (or $\pi^{-1}_{X}(X_{\mathrm{sing}})$) can help us study resolutions of singularities of $X$. More precisely, his idea provides an intrinsic approach to detect the divisors which appear on all smooth models over $X$.

\begin{definition}[Essential Divisors\ /\ Essential Valuations]
Let $E$ be an exceptional prime divisor over $X$. $E$ is called \emph{essential} over $X$ for any resolution $f:Y\to X$, the center of $\mathrm{val}_{E}$ in $Y$ is an irreducible component of $f^{-1}(X_{\mathrm{sing}})$. We say $E$ \emph{appears} in $Y$ if the center of $\mathrm{val}_{E}$ is a divisor in $Y$. The divisorial valuation induced by an essential exceptional prime divisor of $X$ is called an \emph{essential valuation} of $X$. 
\end{definition}

Let $X$ be a topological space, and let $x,y\in X$. We say $y$ is a \emph{generalization} of $x$ (or $x$ is a \emph{specialization} of $y$) if $x\in\overline{\{y\}}$ in $X$, where $\overline{\{ \ \}}$ denotes the closure in $X$.

For the remainder of this section, we let $\mathbb{K}$ denote a field extension of $\mathbb{C}$. 
\begin{proposition}[\cite{IK03}]\label{Nashgen}
Let $X$ be a variety, and let $\alpha: \mathrm{Spec}(\mathbb{K}[\![t]\!])\to X$ be a $\mathbb{K}$-arc such that $\alpha(0)\in X_{\mathrm{sing}}$. Then $\alpha$ has a generalization $\beta\in X_{\infty}$ such that $\beta(0)\in X_{\mathrm{sing}}$ and $\beta(\eta)\notin X_{\mathrm{sing}}$.   
\end{proposition}
\begin{proof}
See \cite[Lemma~2.12]{IK03} or \cite[Chapter 3, Proposition~4.3.2]{CAS18}.    
\end{proof}

\begin{remark}
The proof of this proposition actually needs the condition that $X$ is over a field $k$ of characteristic zero. 
\end{remark}
\begin{proposition}[\cite{IK03}]
$\pi^{-1}_{X}(X_{\mathrm{sing}})$ has the following decomposition
\[
\pi^{-1}_{X}(X_{\mathrm{sing}})=\bigcup_{i\in I}C_{i}, 
\]where the $C_{i}$ are the irreducible components of $\pi^{-1}_{X}(X_{\mathrm{sing}})$, and each $C_{i}$ contains arcs $\alpha\in X_{\infty}$ such that $\alpha(\eta)\notin X_{\mathrm{sing}}$.
\end{proposition}
\begin{proof}
See \cite[Lemma 2.12]{IK03}.   
\end{proof}

\begin{definition}[Nash Components]
These $C_{i}$ are called the \emph{Nash components} of $X$.    
\end{definition}

A $\mathbb{K}$-arc $\alpha:\mathrm{Spec}(\mathbb{K}[\![t]\!])\to X$ induces a discrete semi-valuation 
\begin{equation*}
\begin{split}
\mathrm{val}_{\alpha}:\mathcal{O}_{X,\pi_{X}(\alpha)}&\to\mathbb{Z} \\
f&\mapsto \mathrm{ord}_{t}[f(\alpha(t))]
\end{split}
\end{equation*}
on the local ring $\mathcal{O}_{X,\pi_{X}(\alpha)}$, defined by taking composition with the $t$-adic valuation. 

In general, this semi-valuation cannot be extended to a valuation defined on the function field $\mathbb{C}(X)$. It will be a valuation if $\mathbb{C}(X)\subseteq \mathbb{K}(\!(t)\!)$ is a field extension, which is the same as the image of $\alpha$ being dense in $X$. This leads to the concept of \emph{fat arcs}.

\begin{definition}[Fat Arcs]
An arc $\alpha:\mathrm{Spec}(\mathbb{K}[\![t]\!])\to X$ is called \emph{thin} (or \emph{algebraically thin}) if the image of $\alpha$ is contained in a proper closed subscheme of $X$. Otherwise, $\alpha$ is called \emph{fat} (or \emph{algebraically fat}). Equivalently, $\alpha$ is fat if and only if $\mathrm{val}_{\alpha}$ is a valuation on $\mathbb{C}(X)$.
\end{definition} 
Now we define the semi-valuations of $X$ induced by the Nash components of $X$, which will be shown to actually be valuations on $X$. Let $C_{i}$ be a Nash component, and let $\alpha_{i}\in C_{i}$ be the generic point. By Proposition \ref{Nashgen}, we have a generalization $\tilde{\alpha}_{i}\in \pi_{X}^{-1}(X_{\mathrm{sing}})$ such that $\tilde{\alpha}_{i}(\eta)\notin X_{\mathrm{sing}}$ and $\tilde{\alpha_{i}}(0)\in X_{\mathrm{sing}}$. Hence, we have $C_{i}=\widebar{\{\alpha_{i}\}}\subseteq\overline{\{\tilde{\alpha}_{i}\}}$. Since $C_{i}$ is an irreducible component of $\pi_{X}^{-1}(X_{\mathrm{sing}})$, we get that $\tilde{\alpha}_{i}$ is also the generic point of $C_{i}$ (i.e., $\tilde{\alpha}_{i}=\alpha_{i}$, and this also implies $C_{i}\nsubseteq\pi_{X}^{-1}(X_{\mathrm{sing}})$). 
We obtain the semi-valuation $\mathrm{val}_{C_{i}}:=\mathrm{val}_{\alpha_{i}}$ of $X$.

We now explain how arcs on $X$ with $\alpha(0)\in X_{\mathrm{sing}}$ can be lifted to arcs on a resolution of $X$.
\begin{lemma}
Let $f:Y\to X$ be a resolution of singularities of $X$, and assume that it induces an isomorphism between $X\setminus Z$ and $Y\setminus f^{-1}(Z)$ for some proper closed subset $Z\subsetneqq X$.  
Let $C_{i}^{\circ}$ denote the open subset of a given Nash component $C_{i}$ consisting of arcs $\alpha:Spec(\mathbb{K}[\![t]\!])\to X$ such that $\alpha(\eta)\notin Z$. Then for every $\alpha\in C_{i}^{\circ}$, the arc $\alpha$ can be uniquely lifted to an arc $\tilde{\alpha}:\mathrm{Spec}(\mathbb{K}[\![t]\!])\to Y$.
\end{lemma}
\begin{proof}
Notice that $f$ is an isomorphism outside $Z$, so we have the following commutative diagram.
\begin{center}
\begin{tikzcd}
\mathrm{Spec}(\mathbb{K}(\!( t )\!)) \arrow[r] \arrow[d] & Y \arrow[d, "f"] \\
\mathrm{Spec}(\mathbb{K}[\![t]\!]) \arrow[r, "\alpha"] & X 
\end{tikzcd}
    
\end{center}
Here the downward arrow on the left is induced by the canonical inclusion to the fraction field $\mathbb{K}[\![t]\!]\hookrightarrow\mathbb{K}(\!(t)\!)$, and the top arrow $\mathrm{Spec}(\mathbb{K}(\!(t)\!))\to Y$ is obtained by the condition $\alpha(\eta)\notin Z$.
Notice that $f$ is a proper morphism and $\mathbb{K}[\![t]\!]$ is a valuation ring, so we obtain a unique lifting $\tilde{\alpha}:\mathrm{Spec}(\mathbb{K}[\![t]\!])\to Y$ by the valuative criterion of properness.
\[
\begin{tikzcd}[baseline=-0.8cm]
\mathrm{Spec}(\mathbb{K}(\!( t )\!)) \arrow[r] \arrow[d] & Y \arrow[d, "f"] \\
\mathrm{Spec}(\mathbb{K}[\![t]\!])\arrow[ur, dashed, "\exists\tilde{\alpha}"] \arrow[r, "\alpha"] & X 
\end{tikzcd}
\qedhere
\]
\end{proof}
\begin{remark}
If $f:Y\to X$ is a resolution such that $Y\setminus f^{-1}(X_{\mathrm{sing}})$ is isomorphic to $X\setminus X_{\mathrm{sing}}$ via $f$, we have $Z=X_{\mathrm{sing}}$.    
\end{remark}
\begin{theorem}[\cite{Nash95,IK03,CAS18}]
Let $X$ be a singular variety with singular locus $X_{\mathrm{sing}}$, and let $C$ be a Nash component of $X$. Then there is a unique essential exceptional prime divisor $E$ over $X$ such that
\[
\mathrm{val}_{C}=\mathrm{val}_{E}.
\] In particular, $\mathrm{val}_{C}$ is a valuation and it is essential. Also, there are only finitely many Nash components of $X$.
\end{theorem}

\begin{definition}[Nash Valuations of $X$]
Let $C_{1},\dots,C_{\ell}$ be the Nash components of $X$, then the induced valuations $\mathrm{val}_{C_{1}},\dots,\mathrm{val}_{C_{\ell}}$ are called the \emph{Nash valuations} of $X$.    
\end{definition}

In other words, we have the following map
\begin{equation*}
\begin{split}
\mathcal{N}_{X}:\{\text{Nash components of }X\}&\to\{\text{essential valuations of } X\}\\
C&\longmapsto\mathrm{val}_{C}=\mathrm{val}_{E}
\end{split}    
\end{equation*}
is injective. $\mathcal{N}_{X}$ is called the \emph{Nash map} of $X$. Therefore, we have
\[
\{\text{Nash valuations of }X\}\subseteq\{\text{essential valuations of }X\}.
\]
Consequently, we obtain an intrinsic method to identify some of the essential valuations of $X$ without requiring to consider all resolutions of singularities of $X$.
Now we can formulate \emph{the Nash Problem for $X$}:
\begin{center}
Understand the image of the Nash map $\mathcal{N}_{X}$.    
\end{center}
In particular, we are interested in that if every essential valuation of $X$ corresponds to a Nash valuation of $X$ (i.e., the surjectivity of the Nash map $\mathcal{N}_{X}$).

As mentioned in the introduction, it is known that $\mathcal{N}_{X}$ is surjective in dimensions 1 and 2 \cite{FdBPP12} and several interesting cases in
arbitrary dimensions \cite{IK03,PPP08,GPPD07,BLM25}, but the surjectivity fails in general starting in dimensions 3 \cite{dF13,IK03}. Furthermore, \cite{dFD16,LJR12} provide some characterizations for the image of $\mathcal{N}_{X}$. In particular, T.~de~Fernex and R. Docampo introduce the concept of \emph{terminal valuations} and prove that they belong to the image of $\mathcal{N}_{X}$ \cite{dFD16}.

Since there are no terminal valuations for $X$ having only terminal singularities \cite{dFD16}, we need other approaches to characterize their Nash valuations. In \cite{JK13}, J.~Kollár and J.~M.~Johnson provide some characterizations for singularities of type $cA$ in arbitrary dimensions. Recently, there is a result proven by H.-K.\ Chen \cite{Chen23} for 3-fold terminal singularities of type $cA/m$.
\begin{theorem}[{\cite[Theorem~1.1]{Chen23}}]
Let $(X,p)$ be a germ of 3-fold terminal singularity of type $cA/m$ for $m\in\mathbb{Z}_{>0}$, then there is an one-to-one correspondence
between Nash valuations of $X$ and exceptional prime divisors over $X$
whose discrepancies are less than or equal to one.    
\end{theorem}
On the other hand, the example in \cite[Proposition 4.5]{IK03} also mentions that its Nash valuation of $X$ has minimal discrepancy, which is $1$. These observations serve as evidence that discrepancies might provide a useful criterion to identify whether a given valuation of $X$ is Nash or not.




\section{Nash Valuations for Toric Terminal Singularities}

It is natural to consider the valuations over toric terminal singularities as our initial case. In this section, we will prove that any exceptional prime divisor $E$ over a toric terminal singularity with discrepancy $a(E,X)\leq 1$ induces a Nash valuation. 

We recall some basic setups and characterizations for $\mathbb{Q}$-Gorenstein toric varieties with terminal singularities. Most of the material is from \cite{CLS11} and \cite{Mat02}. In this section, $N$ denotes the standard lattice and $M:=\mathrm{Hom}_{\mathbb{Z}}(N,\mathbb{Z})$ denotes its dual lattice.  
\begin{proposition}
	Let $\Delta$ be a fan in $N_{\mathbb{R}}:=N\otimes_{\mathbb{Z}}\mathbb{R}$, inducing an affine toric variety $X(\Delta)$.
\begin{itemize}
    \item[(1)] Its canonical divisor $K_{X(\Delta)}$ is $\mathbb{Q}$-Cartier (i.e., $X(\Delta)$ is $\mathbb{Q}$-Gorenstein) if and only if for any cone $\tau\in\Delta$, there exists an $m_{\tau}\in M_{\mathbb{Q}}:=M\otimes\mathbb{Q}$ such that $m_{\tau}(v_{i})=\langle m_{\tau},v_{i}\rangle=1$ for every primitive element $v_{i}\in N$ with $\langle v_{i} \rangle\in\Delta$ and $\langle v_{i} \rangle\subset\tau$.
    \item[(2)] Assume $K_{X(\Delta)}$ is $\mathbb{Q}$-Cartier(i.e., $X(\Delta)$ is $\mathbb{Q}$-Gorenstein). Then $X(\Delta)$ has terminal singularities if and only if
\[
\{w\in\tau\cap N|\langle m_{\tau},w\rangle\leq 1\}=\{0\}\cup\{v_{i}\in N| v_{i}\text{ is primitive},  \langle v_{i} \rangle\in\Delta \text{ and } \langle v_{i} \rangle\subset\tau \}
\]for every $\tau\in\Delta$. 
\end{itemize}
\end{proposition} 

Given a cone $\sigma\in\Delta$, we consider the following partial order on $\sigma\cap N$.
\begin{definition}
Let $u_{1},u_{2}\in\sigma\cap N$. We define $u_{1}\leq u_{2}$ if $u_{2}\in u_{1}+\sigma$. i.e., there exists a $u'\in\sigma\cap N$ such that $u_{2}=u_{1}+u'$. For a subset $A\subset\sigma\cap N$, an element $u\in A$ is minimal in $A$ if for every $v\in A$ with $v\leq u$ we have $u=v$.
\end{definition}

Denote $S_{\sigma}:=(\cup_{\tau\preceq\sigma,\tau:\text{singular}}\stackrel{\circ}{\tau})\cap N$ for a given cone $\sigma\in\Delta$, where $\tau\preceq\sigma$ means that $\tau$ is a face of $\sigma$. S.\ Ishii and J.\ Kollár describe the Nash toric valuations on a toric variety via this partial order $\leq$.

\begin{theorem}[\cite{IK03}]\label{IK03main}
Let $X(\Delta)$ be an affine toric variety in any dimension. Then its Nash map is bijective. Moreover, Nash valuations of $X(\Delta)$ (essential valuations of $X(\Delta)$) are in one-to-one correspondence to minimal elements in $S_{\sigma}$ for some cone $\sigma\in\Delta$. 
\end{theorem}
Here, we try characterizing minimal elements with respect to $\leq$ for terminal toric singularities.
\begin{proposition}\label{aleq1min} Let $\Delta$ be a fan associated to a toric variety $X(\Delta)$ of dimension at least $3$, and let $\tau\in\Delta$ be a singular cone which defines a terminal singularity. A primitive element $w\in\stackrel{\circ}{\tau}\cap N$ is minimal in $S_{\tau}$ if $\langle m_{\tau},w\rangle\leq 2$.
\end{proposition}
\begin{proof}
We prove its contrapositive statement. Let $\tau\in\Delta$ be a cone such that $\tau\cap N$ is singular and corresponds to a terminal singularity, and consider an element $w\in\stackrel{\circ}{\tau}\cap N$ which is not minimal in $S_{\tau}$. Hence, there exists a non-zero element $u \in S_{\tau}$ such that $u\leq w$, i.e., $w=u+u'$ for some non-zero $u'\in\tau\cap N$.

Notice that $\langle m_{\tau},v_{i}\rangle=1$ because $X(\tau)$ is $\mathbb{Q}$-Gorenstein, and $\langle m_{\tau},u\rangle>1$  for a non-zero $u\in(\tau\cap N)\setminus\{v_{i}\in N| v_{i}:\text{primitive},  \langle v_{i} \rangle\in\Delta \text{ and } \langle v_{i} \rangle\subset\tau \}$ since the $X(\tau)$ has a terminal singularity.

We have the following two cases:\\
\textbf{Case 1} If $u'\in(\tau\cap N)\setminus(\cup_{i}\langle v_{i}\rangle)$, then we have $\langle m_{\tau},u'\rangle>1$, and
\[
\langle m_{\tau},w\rangle=\langle m_{\tau},u\rangle+\langle m_{\tau},u'\rangle>1+1=2. \
\]
\textbf{Case 2} If $u'\in\cup_{i}\langle v_{i}\rangle$, then we have $u'=\alpha v_{j}$ for some $v_{j}$ and some $\alpha\in\mathbb{Z}_{>0}$, and we get $\langle m_{\tau},u'\rangle=\alpha\langle m_{\tau},v_{j}\rangle=\alpha\cdot 1\geq 1$. So
\[
\langle m_{\tau},w\rangle=\langle m_{\tau},u\rangle+\langle m_{\tau},u'\rangle>1+1=2. \
\]
Therefore, we obtain
\[
\langle m_{\tau},w\rangle>2.
\qedhere
\]
\end{proof}
By this Proposition, Theorem \ref{IK03main} can be reformulated as follows.
\begin{theorem}\label{mainthm}
    Let $X(\Delta)$ be a $\mathbb{Q}$-Gorenstein toric variety of dimension a least 3 having only terminal singularities (hence, we have $a(E,X(\Delta))> 0$ for all $E$ over $X$), and let $E$ be an exceptional prime divisor over $X(\Delta)$ with $a(E,X(\Delta))\leq 1$. Then the induced valuation $\mathrm{val}_{E}$ is a Nash valuation of $X(\Delta)$. 
\end{theorem}
\begin{proof} 

Recall that for every exceptional prime divisor $E$ over a toric variety, its discrepancy has the form $a(E,X(\Delta))=\langle m_{\tau},w\rangle-1$, where $w\in \stackrel{\circ}{\tau}\cap N$ is the primitive element of a ray $\rho_{E}:=\langle w\rangle\subset\stackrel{\circ}{\tau}$ corresponding to $E$. 

By assumption, since $a(E,X(\Delta))=\langle m_{\tau},w\rangle-1\leq 1$, we get $\langle m_{\tau},w\rangle\leq 2$, and we obtain $w\in\stackrel{\circ}{\tau}\cap N$ is minimal in $\stackrel{\circ}{\tau}\cap N$ by Proposition \ref{aleq1min}. Hence, the associated valuation $\text{val}_{E}$ is a Nash valuation by Theorem \ref{IK03main}.  
\end{proof}

This result suggests that we can try using discrepancies to characterize Nash valuations for terminal singularities. In particular, we propose the following conjecture.

\begin{conjecture}[Conjecture A]\label{conj1} Let $(X,p)$ be 3-fold terminal. Then every exceptional prime divisor $E$ with discrepancy $a(X,E)\leq 1$ induces a Nash valuation $\mathrm{val}_{E}$ of $X$.
\end{conjecture}
We also propose a weaker version. 
\begin{conjecture}[Conjecture B]\label{conj2} Let $(X,p)$ be 3-fold terminal. Then every exceptional prime divisor $E$ with minimal discrepancy (i.e., $a(X,E)=\frac{1}{m}$) induces a Nash valuation $\mathrm{val}_{E}$ of $X$.
\end{conjecture}
In the next section, we will prove that for $X$ is a $3$-fold terminal singularity of type cAx/2 Conjecture B is true.


\section{Nash Valuations for 3-Fold Terminal Singularities of type cAx/2}

In this section, we prove that the Conjecture \ref{conjB} holds for $3$-fold terminal singularities of type cAx/2. In other words, the exceptional prime divisors over $3$-fold terminal singularities $X$ of type cAx/2 computing the minimal discrepancy will induce Nash valuations of $X$. 

\begin{theorem}
Let $(X,p)$ be a germ of 3-fold terminal singularity of type cAx/2. Then any prime divisor $E$ over $X$ with minimal discrepancy ($a(E,X)=\frac{1}{2}$ in this case) induces a Nash valuation $\mathrm{val}_{E}$ of $X$.
\end{theorem}

\subsection{Candidates for Nash Valuations}
Let $(X,p)$ be a germ of terminal singularity. It makes sense to choose a prime divisor $E$ over $(X,p)$ with $a(E,X)\leq 1$ to be a candidate to produce a Nash valuation of $X$ because $\mathrm{val}_{E}$ is, in fact, an essential valuation of $X$. 
\begin{proposition}[{\cite[
Proposition 24]{JK13}}]
Let $(X,p)$ be a germ of canonical singularity (see \cite{KM98,Mat02}), and let $E$ be an exceptional prime divisor over $X$ centered at $p$ and satisfy
\[
a(E,X)< 1+(\text{the minimal discrepancy of }(X,p)).
\]
Then $\mathrm{val}_{E}$ is an essential valuation of $X$.    
\end{proposition}
\begin{corollary}
Let $(X,p)$ be a 3-fold terminal singularity with Gorenstein index $m\geq1$. Then any exceptional prime divisor $E$ over $X$ such that $a(E,X)\leq 1$ is essential for $X$. In particular, any exceptional prime divisor over $X$ with minimal discrepancy (in this case, $\frac{1}{m}$) is essential for $X$.    
\end{corollary}

\subsection{Setup for the Proof}\label{sec:NashEqn}
For a singularity $(X,p)$, define the partial order among the divisorial valuations centered at $p\in X$ :
\[
    \mathrm{val}_{E} \leq_{X} \mathrm{val}_{F}
    \quad\text{if}\quad
    \mathrm{val}_{E}(f) \leq \mathrm{val}_{F}(f)
    \quad\text{for all}\quad
    f\in \mathcal{O}_{X,p}.
\]A divisorial valuation on $X$ $\mathrm{val}_{E}$ is \emph{minimal} if it is minimal among all divisorial vlautions centered at $p$ with respect to $\leq_{X}$.

\begin{proposition}[\cite{dFD16,BLM25}]\label{prop:mininash}
\[\{\text{minimal valuations of }X\} \subseteq\{ \text{Nash valuations of }X\}.\]
\end{proposition}

We outline the strategy of the proof. Let $(X,p)$ be a 3-fold terminal singularity with Gorenstein index $m\geq1$, and let $\pi:\widebar{X}\to X$ be a divisorial contraction with minimal discrepancy as in Definition \ref{def:HayDivCon}. Denote by $E=\pi^{-1}(p)$ the $\mathbb{Q}$-Cartier exceptional prime divisor in $\widebar{X}$ obtained by $\pi$. Let $F$ be an essential exceptional prime divisor over $X$ centered at $p$. 

Here we derive an equation relating $\text{val}_{E}$ and $\text{val}_{F}$. Since $\pi:\widebar{X}\to X$ is a divisorial contraction, $\widebar{X}$ has only terminal singularities, which are isolated points (notice that $\mathrm{codim(}\widebar{X}_{\mathrm{sing}},\widebar{X})\geq 3$ since $\widebar{X}$ is terminal). Let $\phi:Y\to\widebar{X}$ be a resolution. Notice that $\mathrm{val}_{F}$ is essential over $X$, so $F$ appears in $Y$. Hence, $F$ is centered at some singular point $q\in \widebar{X}$.

Let $g\in m_{\widebar{X},q}\setminus \mathcal{I}_{E,q}$ be irreducible, where $m_{\widebar{X},q}\subseteq\mathcal{O}_{\widebar{X},q}$ is the maximal ideal, and $\mathcal{I}_{E}$ is the sheaf of ideals associated with $E$. Let $G=\mathrm{div}_{\widebar{X}}(g)\subseteq\widebar{X}$ be the principal divisor induced by $g$, which is prime in a neighborhood of $q\in\widebar{X}$.

Let $h\in\mathcal{I}_{\pi_{*}G,p}$, where $\pi_{*}G$ is the push-forward of $G$ via $\pi$ in $X$, which is a Weil divisor containing $p$, and $\mathcal{I}_{\pi_{*}G}$ is the sheaf of ideals associated with $\pi_{*}G$. Let $H=\mathrm{div}_{X}(h)\subseteq X$ be the principal divisor induced by $h$, which is Cartier. Then 
\[
\pi^{*}H=aG+bE+A 
\] where $a,b\in \mathbb{Z}_{\geq 1}$ and $A$ is some $\mathbb{Q}$-Cartier divisor in $\widebar{X}$ without components from $G$ or $E$ (notice that here $E$ is $\mathbb{Q}$-Cartier, $a\geq 1$ since $h\in\mathcal{I}_{\pi_{*}G}$, and $b\geq 1$ since $x\in\pi_{*}G$).

Applying $\phi^{*}$ to $\pi^{*}H$, we get
\[
\phi^{*}\pi^{*}H=\phi^{*}(aG+bE+A)=a\tilde{G}+b\tilde{E}+(aa_{1}+bc+d)F+(\mathrm{others})  
\]where $\phi^{*}E=\tilde{E}+cF+(\mathrm{others})$, $\phi^{*}G=\tilde{G}+a_{1}F+(\mathrm{others})$, $\phi^{*}A=dF+\tilde{A}+(\mathrm{others})$, with $a_{1}\in\mathbb{Z}_{\geq 1}$, $c\in\mathbb{Q}_{> 0}, d\in\mathbb{Q}_{\geq 0}$. "Others" means divisors whose supports are disjoint from the shown divisors in the corresponding equations. In addition, $\tilde{G}$, $\tilde{E}$, and $\tilde{A}$ denote the strict transforms of $G$, $E$, and $A$ to $Y$.

We can see that
\[
\mathrm{val}_{E}(h)=b,
\quad
\mathrm{val}_{F}(h)=aa_{1}+bc+d.
\]In other words,
\[
\mathrm{val}_{F}(h)=aa_{1}+c\mathrm{val}_{E}(h)+d.\label{eq:valeqn}\tag{\#}
\]Here is the main idea of the proof. Assume that $\mathrm{val}_{E}$ is non-Nash, so $\mathrm{val}_{E}$ is non-minimal by Proposition~\ref{prop:mininash}, which means there is a Nash valuation $\mathrm{val}_{F}$ centered at $p$ such that $\mathrm{val}_{F}<_{X}\mathrm{val}_{E}$. Our goal is to find some particular $h\in\mathcal{I}_{\pi_{*}G,p}$ such that $\text{val}_{E}(h)<\text{val}_{F}(h)$, which leads to a contradiction. To obtain such inequality, it suffices to construct an $h\in\mathcal{I}_{\pi_{*}G,p}$ satisfying either (1) $\mathrm{val}_{E}(h)=1(=b)$ or (2) $c\geq 1$ in equation~(\ref{eq:valeqn}).
\begin{remark}
In the following proof, we will construct an $h\in\mathcal{I}_{\pi_{*}G,p}$ such that $\mathrm{val}_{E}(h)=1$.    
\end{remark}

\subsection{Weighted Blowing-Ups}
From \cite{Hay99,Hay00}, we know that any divisorial contraction with minimal discrepancy is obtained by a \emph{weighted blowing-up} with some weight $\sigma=\frac{1}{m}(a,b,c,d)\in\frac{1}{m}(\mathbb{Z}_{>0})^{4}$. In this section, we introduce the notation and tools that will be used in the proof of our main theorem in this section. The details can be found in \cite{Hay99,Hay00,Chen16,Chen23}.

The cyclic quotient singularity $Y=\mathbb{C}^{4}/\frac{1}{m}(\alpha_{1},\alpha_{2},\alpha_{3},\alpha_{4})$ is the toric variety corresponding to the lattice $\widebar{N}:=\mathbb{Z}e_{1}+\mathbb{Z}e_{2}+\mathbb{Z}e_{3}+\mathbb{Z}e_{4}+\mathbb{Z}e=N+\mathbb{Z}e$ and the cone $C=\mathbb{R}_{\geq 0}e_{1}+\mathbb{R}_{\geq 0}e_{2}+\mathbb{R}_{\geq 0}e_{3}+\mathbb{R}_{\geq 0}e_{4}$, where $e_{i}$, $i=1,2,3,4$, are the standard basis vectors for $\mathbb{Z}^{4}$, and $e=\frac{1}{m}(\alpha_{1},\alpha_{2},\alpha_{3},\alpha_{4})\in\frac{1}{m}\mathbb{Z}^{4}$. The fan $\Delta$ of $Y$ consist of all the faces of $C$. In other words,
\[
Y=\mathbb{C}^{4}/\tfrac{1}{m}(\alpha_{1},\alpha_{2},\alpha_{3},\alpha_{4})\cong\mathrm{Spec}(\mathbb{C}[C^{\vee}\cap \widebar{M}])
\]where $\widebar{M}:=\mathrm{Hom}_{\mathbb{Z}}(\widebar{N},\mathbb{Z})$.

Let us give a description of the weighted blowing-ups for $\mathbb{C}^{4}/\frac{1}{m}(\alpha_{1},\alpha_{2},\alpha_{3},\alpha_{4})$. Let $\sigma=\frac{1}{m}(a,b,c,d)\in \widebar{N}$ with $a,b,c,d>0$, and assume $e_{1},e_{2},e_{3},e_{4},\sigma$ generate the lattice $\widebar{N}$. This $\sigma$ will be called a \emph{weight}. 

We construct the \emph{weighted blowing-up} $\bar{\pi}_{\sigma}:\widebar{Y}\to Y=\mathbb{C}^{4}/\frac{1}{m}(\alpha_{1},\alpha_{2},\alpha_{3},\alpha_{4})$ with weight $\sigma$ as follows. Subdivide the cone $C$ by adding the $1$-dimensional cone $\mathbb{R}_{\geq 0}\sigma$. We obtain a subdivision of $C$ into $4$ cones:
\begin{equation*}
\begin{split}
&C_{1}=\mathbb{R}_{\geq0}\sigma +\mathbb{R}_{\geq0}e_{2}+\mathbb{R}_{\geq0}e_{3}+\mathbb{R}_{\geq0}e_{4},\\
&C_{2}=\mathbb{R}_{\geq0}e_{1} +\mathbb{R}_{\geq0}\sigma+\mathbb{R}_{\geq0}e_{3}+\mathbb{R}_{\geq0}e_{4},\\
&C_{3}=\mathbb{R}_{\geq0}e_{1}+\mathbb{R}_{\geq0}e_{2} +\mathbb{R}_{\geq0}\sigma+\mathbb{R}_{\geq0}e_{4},\\
&C_{4}=\mathbb{R}_{\geq0}e_{1}+\mathbb{R}_{\geq0}e_{2}+\mathbb{R}_{\geq0}e_{3} +\mathbb{R}_{\geq0}\sigma.
\end{split}    
\end{equation*}
Let $\Delta'$ be the fan consisting of all faces of $C_{1},C_{2},C_{3},C_{4}$. Then the \emph{weighted blowing-up} $\widebar{Y}$ of $Y$ is defined as the toric variety corresponding to the lattice $\widebar{N}$ and the fan $\Delta'$, and $\bar{\pi}_{\sigma}:\widebar{Y}\to Y$ is the morphsim induced from the natural map of fans $(\widebar{N},\Delta')\to(\widebar{N},\Delta)$.

The toric variety $\widebar{Y}$ is covered by 4 affine open sets $\widebar{U}_{1},\widebar{U}_{2},\widebar{U}_{3},\widebar{U}_{4}$ corresponding to the cones $C_{1},C_{2},C_{3},C_{4}$, respectively. These affine open sets are as follows:
\begin{equation*}
\begin{split}
&\widebar{U}_{1}=\mathbb{C}^{4}/\tfrac{1}{a}(m,-b,-c,-d)\cong\mathrm{Spec}(\mathbb{C}[C_{1}^{\vee}\cap \widebar{M}]),\\
&\widebar{U}_{2}=\mathbb{C}^{4}/\tfrac{1}{b}(-a,m,-c,-d)\cong\mathrm{Spec}(\mathbb{C}[C_{2}^{\vee}\cap \widebar{M}]),\\
&\widebar{U}_{3}=\mathbb{C}^{4}/\tfrac{1}{c}(-a,-b,m,-d)\cong\mathrm{Spec}(\mathbb{C}[C_{3}^{\vee}\cap \widebar{M}]),\\
&\widebar{U}_{4}=\mathbb{C}^{4}/\tfrac{1}{d}(-a,-b,-c,m)\cong\mathrm{Spec}(\mathbb{C}[C_{4}^{\vee}\cap \widebar{M}]).
\end{split}    
\end{equation*}
And $\bar{\pi}_{\sigma}$ are given by:
\begin{equation*}
\begin{aligned}
&\bar{\pi}_{\sigma}|_{\widebar{U}_{1}}:\widebar{U}_{1}\longrightarrow Y  &\quad
&(x,y,z,u)\mapsto (x^{\frac{a}{m}},x^{\frac{b}{m}}y,x^{\frac{c}{m}}z,x^{\frac{d}{m}}u),\\
&\bar{\pi}_{\sigma}|_{\widebar{U}_{2}}:\widebar{U}_{2}\longrightarrow Y &\quad
&(x,y,z,u)\mapsto (y^{\frac{a}{m}}x,y^{\frac{b}{m}},y^{\frac{c}{m}}z,y^{\frac{d}{m}}u),\\
&\bar{\pi}_{\sigma}|_{\widebar{U}_{3}}:\widebar{U}_{3}\longrightarrow Y &\quad
&(x,y,z,u)\mapsto (z^{\frac{a}{m}}x,z^{\frac{b}{m}}y,z^{\frac{c}{m}},z^{\frac{d}{m}}u),\\
&\bar{\pi}_{\sigma}|_{\widebar{U}_{4}}:\widebar{U}_{4}\longrightarrow Y &\quad
&(x,y,z,u)\mapsto (u^{\frac{a}{m}}x,u^{\frac{b}{m}}y,u^{\frac{c}{m}}z,u^{\frac{d}{m}}).
\end{aligned}
\end{equation*}
In these expressions, $x$, $y$, $z$, $u$ denote coordinates in the ambient $\mathbb{C}^4$. 

The exceptional divisor $\widebar{E}$ of $\bar{\pi}_{\sigma}$ is isomorphic to the weighted projective space $\mathbb{P}_{\mathbb{C}}(a,b,c,d)$.

Let us describe weighted blowing-ups for a subvariety defined by a semi-invariant in $\mathbb{C}^{4}/\frac{1}{m}(\alpha_{1},\alpha_{2},\alpha_{3},\alpha_{4})$. Let $\phi\in\mathbb{C}\{x,y,z,u\}$ be a $\mathbb{Z}_{m}$-semi-invariant. It defines a hyperquotient singularity $X=(\phi=0)\subseteq\mathbb{C}^{4}/\frac{1}{m}(\alpha_{1},\alpha_{2},\alpha_{3},\alpha_{4})$. Let $\widebar{X}=(\bar{\pi}_{\sigma})^{-1}(X)$ be the strict transform of $X$ via $\bar{\pi}_{\sigma}$ and let $\pi_{\sigma}=\bar{\pi}_{\sigma}|_{\widebar{X}}$ be the restriction of $\bar{\pi}_{\sigma}$ in $\widebar{X}$. Then $\pi_{\sigma}:\widebar{X}\to X$ is called the \emph{weighted blowing-up with weight $\sigma=\frac{1}{m}(a,b,c,d)$} (or, briefly, the $\sigma$-blowing-up). Let $U_{i}:=\widebar{U}_{i}\cap\widebar{X}$ for $i=1,2,3,4$, then each $U_{i}$ is a hyperquotient singularity in $\widebar{U}_{i}$ and $\widebar{X}$ is covered by 4 charts $U_{1},U_{2},U_{3},U_{4}$. We denote $E:=\widebar{E}\cap\widebar{X}\subseteq\mathbb{P}_{\mathbb{C}}(a,b,c,d)$. Notice that $E$ is not necessarily irreducible in $\widebar{X}$. 

We can explicitly express the defining equations of $\widebar{X}\cap U_{i}$ for $i=1,2,3,4$. Write a semi-invariant $\phi=\sum_{i,j,k,\ell}a_{ijk\ell}x^{\alpha_{i}}y^{\beta_{j}}z^{\gamma_{k}}u^{\delta_{\ell}}$, and we define 
\[
\mathrm{wt}_{\sigma}(\phi):=\min \left\{ \frac{a\alpha_{i}+b\beta_{j}+c\gamma_{k}+d\delta_{\ell}}{m} \ \middle| \ a_{ijk\ell}\neq 0 \right\}
\]for a given weight $\sigma=\frac{1}{m}(a,b,c,d)\in\widebar{N}$. Then $\widebar{X}\cap U_{i}$ for $i=1,2,3,4$ is defined by 
\begin{equation*}
\begin{split}
&\phi_{1}:=\phi(x^{\frac{a}{m}},x^{\frac{b}{m}}y,x^{\frac{c}{m}}z, x^{\frac{d}{m}}u)x^{-\mathrm{wt}_{\sigma}(\phi)} \text{ in } U_{1},\\
&\phi_{2}:=\phi(y^{\frac{a}{m}}x,y^{\frac{b}{m}},y^{\frac{c}{m}}z, y^{\frac{d}{m}}u)y^{-\mathrm{wt}_{\sigma}(\phi)} \text{ in } U_{2},\\
&\phi_{3}:=\phi(z^{\frac{a}{m}}x,z^{\frac{b}{m}}y,z^{\frac{c}{m}}, z^{\frac{d}{m}}u)z^{-\mathrm{wt}_{\sigma}(\phi)} \text{ in } U_{3},\\
&\phi_{4}:=\phi(u^{\frac{a}{m}}x,u^{\frac{b}{m}}y,u^{\frac{c}{m}}z, u^{\frac{d}{m}})u^{-\mathrm{wt}_{\sigma}(\phi)} \text{ in } U_{4}.\\
\end{split}
\end{equation*}

Let us describe the singular locus $(\widebar{X})_{\mathrm{sing}}$ of $\widebar{X}$. Write 
\[
(\widebar{X})_{\mathrm{sing}}=[(\widebar{X})_{\mathrm{sing}}]_{\mathrm{ind}=1}\bigcup
[(\widebar{X})_{\mathrm{sing}}]_{\mathrm{ind}>1}
\]where $[(\widebar{X})_{\mathrm{sing}}]_{\mathrm{ind}=1}$ ($[(\widebar{X})_{\mathrm{sing}}]_{\mathrm{ind}>1}$, resp.) denotes the singular locus with Gorenstein index $=1$ (Gorenstein index $>1$, resp.). Then we can describe $(\widebar{X})_{\mathrm{sing}}$ in the following way.
\begin{proposition}[2.1 in \cite{Chen16}]\label{prop:Xbarsing}
Keep the notation in this section. Let $\pi_{\sigma}:\widebar{X}\to X$ be the weighted blowing-up of $X$ with respect to the weight $\sigma$, then we have
\begin{itemize}
    \item[(1)] $[(\widebar{X})_{\mathrm{sing}}]_{\mathrm{ind}>1}=\widebar{X}\cap(\mathbb{P}_{\mathbb{C}}(a,b,c,d))_{\mathrm{sing}}=E\cap (\mathbb{P}_{\mathbb{C}}(a,b,c,d))_{\mathrm{sing}}$.
    \item[(2)] If $\widebar{U}_{i}\cong\mathbb{C}^{4}$ for some $i\in\{1,2,3,4\}$. Then $(\widebar{X})_{\mathrm{sing}}\cap U_{i}\subseteq E_{\mathrm{sing}}\cap U_{i}$.
    \item[(3)] If $\widebar{X}$ is a terminal 3-fold, then $[(\widebar{X})_{\mathrm{sing}}]_{\mathrm{ind}=1}\subseteq E_{\mathrm{sing}}$.
\end{itemize}    
\end{proposition}


\section{Proof of Theorem \ref{mainthmcax2}}

We follow the notation of Section \ref{sec:NashEqn}. Assume $\mathrm{val}_{E}$ is not a Nash valuations of $X$, so $\mathrm{val}_{E}$ is non-minimal by Proposition \ref{prop:mininash}. Then there is a divisor $F$ over $X$ inducing a Nash valuation $\mathrm{val}_{F}$ such that 
\[
\mathrm{val}_{F}<_{X}\mathrm{val}_{E},
\]and $F$ must be centered at some singular point of $\widebar{X}$. It suffices to show that for any $q\in\widebar{X}_{\mathrm{sing}}$ such that $F$ is centered at $q$, there is $f\in\mathcal{O}_{X,p}$ such $\mathrm{val}_{F}(f)>\mathrm{val}_{E}(f)$; this contradicts $\mathrm{val}_{F}<_{X}\mathrm{val}_{E}$, concluding that $\mathrm{val}_{E}$ is a Nash valuation of $X$.   

We follow the arguments in \cite[Proposition 18]{Chen16} and \cite[5.2]{HayG05}. Recall that the standard form of this type of 3-fold terminal singularity $(X,p)$ is 
\[
X=(x^{2}+y^{2}+f(z,u)=0)\subseteq\mathbb{C}^{4}/\tfrac{1}{2}(0,1,1,1)
\]where $f(z,u)\in(z,u)^{4}\subseteq\mathbb{C}\{z,u\}$ is a $\mathbb{Z}_{2}$-invariant with respect to $\frac{1}{2}(0,1,1,1)$, and $p=(0,0,0,0)\in X$. Define 
\[
\tau_{0}:=\min\{\,i+j\,|\,z^{i}u^{j}\in\mathrm{supp}(f(z,u))\}\geq 4,
\]where $\mathrm{supp}(f(z,u))$ denotes the set of all monomials with a non-zero coefficient appearing in $f(z,u)$. Notice that $f(z,u)$ is a $\mathbb{Z}_{2}$-invariant with respect to $\frac{1}{2}(0,1,1,1)$, so this implies $\tau_{0}\in\mathbb{Z}_{>0}$ is even. 

Let $f_{\tau_0}(z,u)=\sum_{i+j=\tau_{0}}a_{ij}z^{i}u^{j}$, and express it in the form of the following factorization:
\[
f_{\tau_0}(z,u)=\sum_{i+j=\tau_{0}}a_{ij}z^{i}u^{j}=\prod_{t\in T}(a_{t}z+b_{t}u)^{m_{t}}
\]
where $\sum_{t\in T}m_{t}=\tau_{0}$. 

\subsection{Case 1: \texorpdfstring{$f_{\tau_{0}}(z,u)$}{ft0(z,u)} is not a perfect square}
From \cite[8.2]{Hay99}, consider the weighted blow-up with the following weight
\[
\sigma:=
\begin{cases}
   \tfrac{1}{2}(\frac{\tau_{0}}{2},\frac{\tau_{0}}{2}+1,1,1),&\text{if }\frac{\tau_{0}}{2}\text{ is even},  \\
   \tfrac{1}{2}(\frac{\tau_{0}}{2}+1,\frac{\tau_{0}}{2},1,1),&\text{if }\frac{\tau_{0}}{2}\text{ is odd}. 
\end{cases}
\]
We consider the first case $\sigma=\frac{1}{2}(\frac{\tau_{0}}{2},\frac{\tau_{0}}{2}+1,1,1)$, and the second case follows by an analogous argument. The weighted blowing-up with respect to $\sigma$ induces a divisorial contraction $\pi_{\sigma}:\widebar{X}\to X$ with exceptional prime divisor $E=\pi^{-1}_{\sigma}(p)\subseteq\widebar{X}$ with minimal discrepancy (in this case $a(E,X)=\frac{1}{2}$) over $X$ by Theorem 8.4 in \cite{Hay99}. Notice that $\widebar{X}$ is covered by four affine open sets $U_{1}$, $U_{2}$, $U_{3}$, and $U_{4}$. By \cite[Proposition 18]{Chen16} and \cite[5.2]{HayG05}, we have that $\widebar{X}$ is non-singular on $U_{1}$, it has one terminal quotient singularity of index $\frac{\tau_{0}}{2}+1$ in $U_{2}$, denoted by $Q_{2}$, and several Gorenstein terminal singularities of type $cD$ in $U_{3}$ and $U_{4}$. 

In $U_{2}$, there is a terminal quotient singularity $Q_{2}=(0,0,0,0)\in U_{2}$ with index $\frac{\tau_{0}}{2}+1$ (see \cite[Case 1 of Proposition 18]{Chen16} and \cite[5.2]{HayG05}). We consider $g=x^{2}y^{\frac{\tau_{0}}{2}-1}+y^{\frac{\tau_{0}}{2}}+z^{2}+u^{2}\in\mathcal{O}_{\widebar{X},Q_{2}}\setminus\mathcal{I}_{E,Q_{2}}$ ($g\notin\mathcal{I}_{E,Q_{2}}$ since $E=(x^{2}+f_{\tau_{0}}(z,u)=0)$ in $U_{2}$), and let $G=\mathrm{div}_{\widebar{X}}(g)\subseteq\widebar{X}$. 
Consider $h=x^{2}+y^{2}+z^{2}+u^{2}\in\mathcal{O}_{X,p}$, and let $H=\mathrm{div}_{X}(h)$. 

Here we compute $\pi_{\sigma}^{*}H$, the pullback of $H$ via $\pi_{\sigma}$, in $U_{2}$. Notice that $\pi_{\sigma}$ is the weighted blowing-up with respect to $\sigma$ given by
\[
(\pi_{\sigma}^{\#}(x), \pi_{\sigma}^{\#}(y), \pi_{\sigma}^{\#}(z), \pi_{\sigma}^{\#}(u))=(xy^{\frac{\tau_{0}}{4}},y^{\frac{\tau_{0}}{4}+\frac{1}{2}},zy^{\frac{1}{2}},uy^{\frac{1}{2}}).
\]Hence, 
\begin{align*}
\pi_{\sigma}^{\#}(h)&=\pi_{\sigma}^{\#}(x^{2}+y^{2}+z^{2}+u^{2})=x^{2}y^{\frac{\tau_{0}}{2}}+y^{\frac{\tau_{0}}{2}+1}+z^{2}y+u^{2}y
\\
&=y(x^{2}y^{\frac{\tau_{0}}{2}-1}+y^{\frac{\tau_{0}}{2}}+z^{2}+u^{2})=yg.
\end{align*}
In other words, we obtain $
\pi_{\sigma}^{*}H=G+E$. Notice that $(\pi_{\sigma})_{*}G=H$, so we have $h=x^{2}+y^{2}+z^{2}+u^{2}\in\mathcal{I}_{(\pi_{\sigma})_{*}G,p}$. Applying the equation (\ref{eq:valeqn}) to $\pi_{\sigma}$, we have $b=\text{val}_{E}(h)=1$, so we have $\text{val}_{E}(h)<\text{val}_{F}(h)$ for any divisor $F$ over $X$ centered at $Q_{2}\in\widebar{X}$.

We compute the strict transform of the defining equation of $\widebar{X}$ in $U_{3}$, and it is given by
\[
(\pi_{\sigma}^{\#}(x), \pi_{\sigma}^{\#}(y), \pi_{\sigma}^{\#}(z), \pi_{\sigma}^{\#}(u))=(xz^{\frac{\tau_{0}}{4}},yz^{\frac{\tau_{0}}{4}+\frac{1}{2}},z^{\frac{1}{2}},uz^{\frac{1}{2}}).
\]Hence, in $U_{3}$,
\begin{align*}
\pi_{\sigma}^{\#}(x^{2}+y^{2}+f(z,u))&=\pi_{\sigma}^{\#}(x^{2}+y^{2}+f_{\tau_{0}}(z,u)+f_{>\tau_{0}}(z,u))\\
&=x^{2}z^{\frac{\tau_{0}}{2}}+y^{2}z^{\frac{\tau_{0}}{2}+1}+f_{\tau_{0}}(z^{\frac{1}{2}},uz^{\frac{1}{2}})+f_{>\tau_{0}}(z^{\frac{1}{2}},uz^{\frac{1}{2}})\\
&=x^{2}z^{\frac{\tau_{0}}{2}}+y^{2}z^{\frac{\tau_{0}}{2}+1}+\prod_{t\in T}(a_{t}z^{\frac{1}{2}}+b_{t}uz^{\frac{1}{2}})^{m_{t}}+f_{>\tau_{0}}(z^{\frac{1}{2}},uz^{\frac{1}{2}})\\
&=z^{\frac{\tau_{0}}{2}}\left(x^{2}+y^{2}z+\prod_{t\in T}(a_{t}+b_{t}u)^{m_{t}}+\frac{f_{>\tau_{0}}(z^{\frac{1}{2}},uz^{\frac{1}{2}})}{z^{\frac{\tau_{0}}{2}}}\right).
\end{align*}
We obtain the strict transform of $X=(x^{2}+y^{2}+f(z,u)=0)\subseteq\mathbb{C}/\frac{1}{2}(0,1,1,1)$ in $U_{3}$ is $(x^{2}+y^{2}z+\prod_{t\in T}(a_{t}+b_{t}u)^{m_{t}}+\frac{f_{>\tau_{0}}(z^{\frac{1}{2}},uz^{\frac{1}{2}})}{z^{\frac{\tau_{0}}{2}}}=0)\subseteq\mathbb{C}^{4}$, which contains some Gorenstein terminal singularities of type $cD$. 

Here we compute $(\widebar{X})_{\mathrm{sing}}$ in $U_{3}$ by applying the Jacobian criterion to 
\[
(x^{2}+y^{2}z+\prod_{t\in T}(a_{t}+b_{t}u)^{m_{t}}+\frac{f_{>\tau_{0}}(z^{\frac{1}{2}},uz^{\frac{1}{2}})}{z^{\frac{\tau_{0}}{2}}}=0)\subseteq\mathbb{C}^{4}.
\]Let $F=x^{2}+y^{2}z+\prod_{t\in T}(a_{t}+b_{t}u)^{m_{t}}+\frac{f_{>\tau_{0}}(z^{\frac{1}{2}},uz^{\frac{1}{2}})}{z^{\frac{\tau_{0}}{2}}}$, and we solve the equation $\mathrm{grad}(F)=(0,0,0,0)$ for $(x,y,z,u)$. 
\begin{equation*}
\begin{split}
2x&=0,\\
2yz&=0,\\
y^{2}+\frac{\partial}{\partial z}\left(\frac{f_{>\tau_{0}}(z^{\frac{1}{2}},uz^{\frac{1}{2}})}{z^{\frac{\tau_{0}}{2}}}\right)&=0,\\
\frac{\partial}{\partial u}\left(\prod_{t\in T}(a_{t}+b_{t}u)^{m_{t}}+\frac{f_{>\tau_{0}}(z^{\frac{1}{2}},uz^{\frac{1}{2}})}{z^{\frac{\tau_{0}}{2}}}\right)&=0.   
\end{split}
\end{equation*}
Notice that we have $z=0$ in $U_{3}$ because $(\widebar{X})_{\mathrm{sing}}\cap U_{3}\subseteq E_{\mathrm{sing}}\cap U_{3}$ by Proposition \ref{prop:Xbarsing}. Some direct calculations show that $(\widebar{X})_{\mathrm{sing}}$ in $U_{3}$ is $\{(0,\beta_{t},0,-\frac{a_{t}}{b_{t}})\}_{t\in T}\subseteq\mathbb{C}^{4}$ for some $\beta_{t}\in\mathbb{C}$. 

Consider $h=x^{2}+y^{2}-(\frac{a_{t}}{b_{t}})^{2}z^{2}+u^{2}$ ($H=\mathrm{div}_{X}(h)$), and $g=x^{2}z^{\frac{\tau_{0}}{2}-1}+y^{2}z^{\tau_{0}}+u^{2}-(\frac{a_{t}}{b_{t}})^{2}$ ($G=\mathrm{div}_{\widebar{X}}(g)$). We have
\begin{align*}
\pi_{\sigma}^{\#}(h)&=\pi_{\sigma}^{\#}(x^{2}+y^{2}-(\frac{a_{t}}{b_{t}})^{2}z^{2}+u^{2})=x^{2}z^{\frac{\tau_{0}}{2}}+y^{2}z^{\frac{\tau_{0}}{2}+1}-(\frac{a_{t}}{b_{t}})^{2}z+u^{2}z)\\
&=z(x^{2}z^{\frac{\tau_{0}}{2}-1}+y^{2}z^{\tau_{0}}-(\frac{a_{t}}{b_{t}})^{2}+u^{2})=zg.    
\end{align*}
Hence, we have $
\pi_{\sigma}^{*}H=G+E$. Applying equation (\ref{eq:valeqn}) to $\pi_{\sigma}$, we have $b=\text{val}_{E}(h)=1$, so we have $\text{val}_{E}(h)<\text{val}_{F}(h)$ for any divisor $F$ over $X$ centered at any point in $\{(0,\beta_{t},0,-\frac{a_{t}}{b_{t}})\}_{t\in T}$.

In $U_{4}$, we have an argument similar to the case in $U_{3}$. Therefore, $\text{val}_{E}$ induces a Nash valuation of $X$. 

\subsection{Case 2: \texorpdfstring{$f_{\tau_{0}}(z,u)$}{ft0(z,u)} is a perfect square}
Write $f_{\tau_{0}}(z,u)=(p_{\frac{\tau_{0}}{2}}(z,u))^{2}$. By \cite[Theorem 8.8]{Hay99}, there are two non-isomorphic divisorial contractions with minimal discrepancy for 3-fold terminal singularity of type cAx/2, and they are described by the following ways. 

we apply the changes of coordinates in \cite[8.6]{Hay99}:
\[
x\mapsto x\pm p_{\frac{\tau_{0}}{2}}(z,u),\quad y\mapsto y,\quad z\mapsto z,\quad u\mapsto u.
\]
and $(X,p)$ is rewritten as
\[
(x^{2}\pm 2xp_{\frac{\tau_{0}}{2}}(z,u)+y^{2}+f_{ \geq\tau_{0}+1}(z,u)=0)\subseteq\mathbb{C}^{4}/\tfrac{1}{2}(0,1,1,1).
\]where $\pm$ denotes two changes of variables corresponding to the two non-isomorphic divisorial contractions.

 We deal with the case $X=(x^{2}+ 2xp_{\frac{\tau_{0}}{2}}(z,u)+y^{2}+f_{ \geq\tau_{0}+1}(z,u)=0)$ here; case $X=(x^{2}- 2xp_{\frac{\tau_{0}}{2}}(z,u)+y^{2}+f_{ \geq\tau_{0}+1}(z,u)=0)$ is similar.

From \cite[8.2]{Hay99}, consider the weighted blowing-up with respect to the following weight.
\[
\sigma:=
\begin{cases}
   \frac{1}{2}(\frac{\tau_{0}}{2}+2,\frac{\tau_{0}}{2}+1,1,1),&\text{if }\frac{\tau_{0}}{2}\text{ is even},  \\
   \frac{1}{2}(\frac{\tau_{0}}{2}+1,\frac{\tau_{0}}{2}+2,1,1),&\text{if }\frac{\tau_{0}}{2}\text{ is odd}.
\end{cases}
\]
We consider the first case $\sigma=\frac{1}{2}(\frac{\tau_{0}}{2}+2,\frac{\tau_{0}}{2}+1,1,1)$; the second case $\frac{1}{2}(\frac{\tau_{0}}{2}+1,\frac{\tau_{0}}{2}+2,1,1)$ will follow to an analogous argument. The weighted blowing-up with respect to $\sigma$ induces a divisorial contraction $\pi_{\sigma}:\widebar{X}\to X$ with the exceptional prime divisor $E=\pi^{-1}_{\sigma}(p)\subseteq\widebar{X}$ with minimal discrepancy (in this case $a(E,X)=\frac{1}{2}$) over $X$ by \cite[Theorem 8.7]{Hay99}.

$\widebar{X}$ is non-singular on $U_{2}$. In $U_{1}$, we have that $\widebar{X}$ is singular at $Q_{1}\in U_{1}$ (see \cite[Case 2 of Proposition 18]{Chen16} and \cite[5.2]{HayG05}), which is a terminal quotient singularity of Gorenstein index $\frac{\tau_{0}}{2}+2$. We consider $g=x^{\frac{\tau_{0}}{2}+1}+y^{2}x^{\frac{\tau_{0}}{2}}+z^{2}+u^{2}$, and let $G=\text{div}_{\widebar{X}}(g)\subseteq\widebar{X}$. Consider $h=x^{2}+y^{2}+z^{2}+u^{2}$, and let $H=\mathrm{div}_{X}(h)\subseteq X$. 

We compute $\pi_{\sigma}^{*}H$, and it is given by
\[
(\pi_{\sigma}^{\#}(x), \pi_{\sigma}^{\#}(y), \pi_{\sigma}^{\#}(z), \pi_{\sigma}^{\#}(u))=(x^{\frac{\tau_{0}}{4}+1},yx^{\frac{\tau_{0}}{4}+\frac{1}{2}},zx^{\frac{1}{2}},ux^{\frac{1}{2}}).
\]Hence, we have
\begin{align*}
\pi_{\sigma}^{\#}(h)&=\pi_{\sigma}^{\#}(x^{2}+y^{2}+z^{2}+u^{2})=x^{\frac{\tau_{0}}{2}+2}+y^{2}x^{\frac{\tau_{0}}{2}+1}+z^{2}x+u^{2}x\\
&=x(x^{\frac{\tau_{0}}{2}+1}+y^{2}x^{\frac{\tau_{0}}{2}}+z^{2}+u^{2})=xg.
\end{align*}
We obtain $\pi_{\sigma}^{*}H=G+E
$. Notice that $(\pi_{\sigma})_{*}(G)=H$ ($G$ is the strict transform of $H$), so we apply the equation (\ref{eq:valeqn}) to $\pi_{\sigma}$ and we have $b=\text{val}_{E}(h)=1$. So we have $\text{val}_{E}(h)<\text{val}_{F}(h)$ for any divisor $F$ over $X$ centered at $Q_{1}\in U_{1}\subseteq\widebar{X}$.

We compute the strict transform of the defining equation of $\widebar{X}$ in $U_{3}$, and it is given by
\[
(\pi_{\sigma}^{\#}(x), \pi_{\sigma}^{\#}(y), \pi_{\sigma}^{\#}(z), \pi_{\sigma}^{\#}(u))=(xz^{\frac{\tau_{0}}{4}+1},yz^{\frac{\tau_{0}}{4}+\frac{1}{2}},z^{\frac{1}{2}},uz^{\frac{1}{2}}). 
\]Hence, in $U_{3}$,
\begin{align*}
&\pi_{\sigma}^{\#}(x^{2}+2xp_{\frac{\tau_{0}}{2}}(z,u)+y^{2}+f_{>\tau_{0}+1}(z,u))\\
&\qquad=x^{2}z^{\frac{\tau_{0}}{2}+2}+2xz^{\frac{\tau_{0}}{4}+1}p_{\frac{\tau_{0}}{2}}(z^{\frac{1}{2}},uz^{\frac{1}{2}})+y^{2}z^{\frac{\tau_{0}}{2}+1}+f_{>\tau_{0}+1}(z^{\frac{1}{2}},uz^{\frac{1}{2}})\\
&\qquad=z^{\frac{\tau_{0}}{2}+1}\left(x^{2}z+2xp_{\frac{\tau_{0}}{2}}(1,u)+y^{2}+\frac{f_{>\tau_{0}+1}(z^{\frac{1}{2}},uz^{\frac{1}{2}})}{z^{\frac{\tau_{0}}{2}+1}}\right).
\end{align*}
We obtain that the strict transform of the defining equation $(x^{2}+y^{2}+f(z,u)=0)\subseteq\mathbb{C}/\frac{1}{2}(0,1,1,1)$ in $U_{3}$ is $(x^{2}z+2xp_{\frac{\tau_{0}}{2}}(1,u)+y^{2}+\frac{f_{>\tau_{0}+1}(z^{\frac{1}{2}},uz^{\frac{1}{2}})}{z^{\frac{\tau_{0}}{2}+1}}=0)\subseteq\mathbb{C}^{4}$, which contains Gorenstein terminal singularities at worst of type $cD$ (meaning either non-singular or singular of type $cD$).

Here we compute $(\widebar{X})_{\mathrm{sing}}$ in $U_{3}$ by applying the Jacobian criterion to
\[
(x^{2}z+2xp_{\frac{\tau_{0}}{2}}(1,u)+y^{2}+\frac{f_{>\tau_{0}+1}(z^{\frac{1}{2}},uz^{\frac{1}{2}})}{z^{\frac{\tau_{0}}{2}+1}}=0)\subseteq\mathbb{C}^{4}.
\]Let $F=x^{2}z+2xp_{\frac{\tau_{0}}{2}}(1,u)+y^{2}+\frac{f_{>\tau_{0}+1}(z^{\frac{1}{2}},uz^{\frac{1}{2}})}{z^{\frac{\tau_{0}}{2}+1}}$, and we solve equation $\mathrm{grad}(F)=(0,0,0,0)$ for $(x,y,z,u)$. 
\begin{equation*}
\begin{split}
2xz+2p_{\frac{\tau_{0}}{2}}(1,u)&=0,\\
2y&=0,\\
x^{2}+\frac{\partial}{\partial z}\left(\frac{f_{>\tau_{0}+1}(z^{\frac{1}{2}},uz^{\frac{1}{2}})}{z^{\frac{\tau_{0}}{2}+1}}\right)&=0,\\
\frac{\partial}{\partial u}\left(2xp_{\frac{\tau_{0}}{2}}(1,u)+\frac{f_{>\tau_{0}+1}(z^{\frac{1}{2}},uz^{\frac{1}{2}})}{z^{\frac{\tau_{0}}{2}+1}}\right)&=0.   
\end{split}
\end{equation*}
Notice that we have $z=0$ in $U_{3}$ because $(\widebar{X})_{\mathrm{sing}}\cap U_{3}\subseteq E_{\mathrm{sing}}\cap U_{3}$ by Proposition \ref{prop:Xbarsing}. Some direct calculations show that $(\widebar{X})_{\mathrm{sing}}$ in $U_{3}$ is $\{(a_{t},0,0,b_{t})\}_{t\in T}$ for some $a_{t},b_{t}\in\mathbb{C}$. Consider $h=x^{2}+y^{2}-b_{t}^{2}z^{2}+u^{2}$ ($H=\mathrm{div}_{X}(h)\subseteq X$), and $g=x^{2}z^{\frac{\tau_{0}}{2}-1}+y^{2}z^{\tau_{0}}+u^{2}-b_{t}^{2}\in\mathcal{O}_{\widebar{X},(a_{t},0,0,b_{t})}\setminus\mathcal{I}_{E,(a_{t},0,0,b_{t})}$ ($G=\mathrm{div}_{\widebar{X}}(g)\subseteq\widebar{X}$). We have
\begin{align*}
\pi_{\sigma}^{\#}(h)&=\pi_{\sigma}^{\#}(x^{2}+y^{2}-b_{t}^{2}z^{2}+u^{2})=x^{2}z^{\frac{\tau_{0}}{2}}+y^{2}z^{\frac{\tau_{0}}{2}+1}-b_{t}^{2}z+u^{2}z\\
&=z(x^{2}z^{\frac{\tau_{0}}{2}-1}+y^{2}z^{\tau_{0}}-b_{t}^{2}+u^{2})=zg.    
\end{align*}
Hence, we have $\pi_{\sigma}^{*}H=G+E
$. Notice that $(\pi_{\sigma})_{*}(G)=H$ ($G$ is the strict transform of $H$). So we apply the equation (\ref{eq:valeqn}) to $\pi_{\sigma}$ and we have $b=\text{val}_{E}(h)=1$, so we have $\text{val}_{E}(h)<\text{val}_{F}(h)$ for any divisor $F$ over $X$ centered at any point in $\{(a_{t},0,0,b_{t})\}_{t\in T}$.

In $U_{4}$, we have an argument similar to the case in $U_{3}$. Hence, $\text{val}_{E}$ induces a Nash valuation. Therefore, we prove that any exceptional prime divisor $E$ over $X$ with minimal discrepancy (i.e., $a(E,X)=\frac{1}{2}$) induces a Nash valuation $\mathrm{val}_{E}$ for a 3-fold terminal singularity of type cAx/2.


\section{Nash Valuations for 3-Fold Gorenstein Terminal Singularities}

Let $(X,p)$ be a germ of 3-fold Gorenstein terminal singularity. In this final section, we establish a partial result in support of Conjectures A and B, and provide some examples as evidence (notice that Conjectures A and B are the same because the minimal discrepancy of a 3-fold Gorenstein terminal singularity is $1$ by \cite[Theorem 0.1]{Mar96}).


We use the set-up developed in \ref{sec:NashEqn}, and the result follows directly.
\begin{theorem}\label{ch8prop}
Let $(X,p)$ be a 3-fold Gorenstein terminal singularity. Let $\pi:\widebar{X}\to X$ be a divisorial contraction with minimal discrepancy (it is $1$ in this case), and let $E=\pi^{-1}(p)$ be its exceptional prime divisor. Assume that the singularities of $\widebar{X}$ are terminal cyclic quotient or Gorenstein terminal. Then there is an exceptional prime divisor over $X$ with discrepancy $1$ that induces a Nash valuation. In other words, 
\[
\{F \text{ exceptional over }X\ |\ a(F,X)=1 \text{ and } \mathrm{val}_{F} \text{ is a Nash valuation of } X\}\neq\emptyset.
\]Moreover, if every exceptional prime divisor over $X$ computing the minimal discrepancy can be obtained by a divisorial contraction with minimal discrepancy to $X$, then $\mathrm{val}_{E}$ is a Nash valuation of $X$.
\end{theorem}
\begin{proof}
If $\mathrm{val}_{E}$ is Nash for $X$, then we are done. If not, then there is an exceptional prime divisor $F$ over $X$ such that $\mathrm{val}_{F}$ is Nash for $X$, and $\mathrm{val}_{F}<_{X}\mathrm{val}_{E}$ by Proposition~\ref{prop:mininash}.

Let $\phi:Y\to\widebar{X}$ be a resolution (notice that $F$ appears in $Y$ since $\mathrm{val}_{F}$ is Nash for $X$, which is essential for $X)$. Consider $\phi^{*}E=\tilde{E}+c F+\sum c_{i}F_{i}$ where $\tilde{E}$ denotes the strict transform of $E$ via $\phi$. By direct computations, we have
\[
a(F,X)=a(F,\widebar{X})+c.
\]Notice that $a(F,X)\in\mathbb{Z}_{>0}$ since $(X,p)$ is Gorenstein. 

We claim that this valuation $\mathrm{val}_{F}$ has $a(F,X)=1$. Since $\mathrm{val}_{F}$ is Nash for $X$ (so $\mathrm{val}_{F}$ is essential for $X$), $F$ must be centered at some terminal cyclic quotient singularity or some Gorenstein terminal singularity in $\widebar{X}$.

If $F$ is centered at any Gorenstein point $q\in\widebar{X}$, we have $a(F,\widebar{X})\in\mathbb{Z}_{>0}$. This implies $c=a(F,X)-a(F,\widebar{X})\in\mathbb{Z}_{>0}$ (notice that $c\in\mathbb{Q}_{>0}$ since $F$ is centered at $q\in\widebar{X}$). Applying a similar argument as \ref{sec:NashEqn} to this scenario, we obtain
\[
\mathrm{val}_{F}(h)=aa_{1}+c\mathrm{val}_{E}(h)+d
\]
for some $h\in\mathcal{O}_{X,p}$. Since $c\geq1$, we get $\mathrm{val}_{E}(h)<\mathrm{val}_{F}(h)$, which contradicts the fact $\mathrm{val}_{E}>_{X}\mathrm{val}_{F}$. Hence, $F$ must be centered at some terminal cyclic quotient singularity in $\widebar{X}$. We have
\[
0<a(F,\widebar{X})<1
\]because this singularity has a resolution such that all exceptional prime divisors have discrepancies less than $1$ (see \cite[Proposition 5.1]{Hay99}, and the corresponding morphism is called \emph{economic resolution}). In this case, we must have $a(F,X)=1$ because $a(F,X)\geq 2$ implies 
\[
c=a(F,X)-a(F,\widebar{X})>2-1=1,
\]
which also provides a $h\in\mathcal{O}_{X,p}$ such that $\mathrm{val}_{E}(h)<\mathrm{val}_{F}(h)$. Therefore, there is an exceptional prime divisor over $X$ with discrepancy equal to $1$ such that it induces a Nash valuation of $X$. 

For the second assertion, suppose $\mathrm{val}_{E}$ is non-Nash for $X$, and assume that every exceptional prime divisor over $X$ computing the minimal discrepancy corresponds to a divisorial contraction with minimal discrepancy to $X$. This implies that every exceptional prime divisor over $X$ computing the minimal discrepancy can be obtained by \emph{w-morphisms} to $p\in X$ (see \cite[Definition 2.12]{ChenMRT24}). A similar argument shows that $\mathrm{val}_{F}<_{X}\mathrm{val}_{E}$ for some Nash valuation $\mathrm{val}_{F}$ of $X$ with $a(F,X)=1$.  By \cite[Corollary 2.17]{ChenMRT24}, there is a $u\in\mathcal{O}_{X,p}$ such that $\mathrm{val}_{F}(u)>\mathrm{val}_{E}(u)$, which leads to a contradiction. Hence, $\mathrm{val}_{E}$ is a Nash valuation of $X$ under this additional assumption.

\end{proof}

\subsection{Examples}
Here we select some examples of 3-fold Gorenstein terminal singularities that satisfy the assumptions of the Proposition \ref{ch8prop}.

Define $\tau(f):=\min\{j+k\ |\ z^{j}u^{k} \in \mathrm{supp(}f(z,u))\text{ as a monimial}\ \}$ for $f\in\mathbb{C}\{z,u\}$. Recall that a Gorenstein terminal singularity $(X,p)$ is called of \emph{type $cE$} if it has the following expression.
\[
(x^{2}+y^{3}+yg(z,u)+h(z,u)=0)\subseteq\mathbb{C}^{4},
\]for some $g,h\in\mathbb{C}\{z,u\}$ with $\tau(g)\geq 3$ and $\tau(h)\geq 4$. Furthermore,
\begin{itemize}
    \item[(1)] it is called type $cE_{6}$ if $\tau(h)=4$ and $\tau(g)\geq 3$,
    \item[(2)] it is called type $cE_{7}$ if $\tau(h)\geq 5$ and $\tau(g)= 3$, and
    \item[(3)] it is called type $cE_{8}$ if $\tau(h)=5$ and $\tau(g)\geq 4$.
\end{itemize}
\begin{example}[\cite{ChecE15,Chen16}]
Let $X$ be the Gorenstein terminal singularity of type $cE_{6}$ which satisfies the conditions in \cite[Subcase 1-2 of Theorem 34]{Chen16}. After applying the change of variables mentioned in \cite[Subcase 1-2 of Theorem 34]{Chen16} for $(X,p)$, consider the weighted blowing-up $\pi_{\sigma}:\widebar{X}\to X$ with the weight $\sigma=(3,2,1,1)$. By the argument in \cite[Lemma 7 and Theorem 5]{Chen16}, $\pi_{\sigma}$ is a divisorial contraction with minimal discrepancy to $X$ (in this case it is $1$), and we obtain $(\widebar{X})_{\mathrm{sing}}=[(\widebar{X})_{\mathrm{sing}}]_{\mathrm{ind}>1}=\{Q_{1},Q_{2}\}$ where $Q_{1},Q_{2}$ are terminal quotient singularities of Gorenstein index $3,2$, respectively. Hence, there is a Nash valuation induced by some exceptional prime divisor over $X$ with minimal discrepancy due to Theorem \ref{ch8prop}. 

In fact, there are exactly two divisors over $X$ with discrepancy $1$ (see \cite[Level 1 in 5.4]{ChecE15}). By \cite[Level 1 in 5.4]{ChecE15}, we have $F=\pi_{\sigma'}^{-1}(p)$ with $\sigma'=(2,2,1,1)$, and it is also a divisorial contractions with minimal discrepancy to $X$. Hence, $\mathrm{val}_{E}$ is a Nash valuation of $X$.
\end{example}

\begin{example}[\cite{Chen16}]
Let $(X,p)$ the Gorenstein terminal singularity of type $cE_{6}$ which satisfies the conditions in \cite[Case 3 of Theorem 34]{Chen16}. Consider the weighted blowing-up $\pi_{\sigma}:\widebar{X}\to X$ with the weight $\sigma=(4,3,2,1)$. By the argument in \cite[Case 3 of Theorem 34]{Chen16}, we have $[(\widebar{X})_{\mathrm{sing}}]_{\mathrm{ind}=1}$ are at worst of type $cE_{6}$ and $[(\widebar{X})_{\mathrm{sing}}]_{\mathrm{ind}>1}$ contains three cyclic quotient singularities. Hence, there is a Nash valuation induced by some exceptional prime divisor over $
X$ with minimal discrepancy due to Proposition \ref{ch8prop}.     
\end{example}

There are other examples of 3-fold Gorenstein terminal singularities of type $cE_{7}$ and $cE_{8}$ that satisfy the assumptions in Proposition \ref{ch8prop}. See \cite[Theorem~36 and Theorem~37]{Chen16}.





\subsection*{Acknowledgments}

The author would like to thank his advisor Prof.\ Roi Docampo for many useful ideas, discussions, suggestions, and feedback for this research. He also thanks Chih-Kuang Lee, Shravan Saoji, Kui-Yo Chen, Prof.\ Tsung-Ju Lee, and Weichung Chen for their useful discussions about combinatorics and birational geometry. The author thanks Prof.\ Jungkai Alfred Chen for inviting him to give a talk and having discussions in the algebraic geometry seminar at National Center for Theoretical Science (NCTS) in Taipei, Taiwan. He also thanks Prof.\ Hsueh-Yung Lin for providing some ideas on this topic. The author thanks Le Centre Henri Lebesgue for inviting him to give a talk about this research in \emph{Spring School/Conference: Singularity in Algebraic Geometry} in Rennes, France. He thanks Prof.\ David Bourqui, Christopher Chiu and Sasha Viktorova for the discussions about my results and other related topics. 

\nocite{Sho93}

\bibliographystyle{alpha-links}
\bibliography{refs}

\end{document}